\documentclass[a4paper,11pt]{article}
\usepackage[active]{srcltx}
\usepackage{array}
\usepackage{theorem}
\usepackage{amsmath,amscd,amssymb}
\usepackage{latexsym}
\usepackage{epsfig}
\usepackage{xypic}
\usepackage{setspace}
\theorembodyfont{\sl}

\newtheorem{lemma}{Lemma}[section]

\newtheorem{proposition}[lemma]{Proposition}
\newtheorem{theorem}[lemma]{Theorem}
\newtheorem{corollary}[lemma]{Corollary}

\newcommand{\CC}{\mathbb C}
\newcommand{\HH}{\mathbb H}

\newcommand{\NN}{\mathbb N}

\newcommand{\QQ}{\mathbb Q}

\newcommand{\ZZ}{\mathbb Z}

\newcommand{\latt}[1]{{\langle{#1}\rangle}}

\newcommand{\qedsymbol}{\mbox{$\Box$}}
\newcommand{\qed}{\unskip\nobreak\hfil\penalty50\hskip1em\hbox{}\nobreak
\hfill\qedsymbol\parfillskip=0pt\finalhyphendemerits=0}
\newenvironment{proof}{\begin{ProofwCaption}{Proof}}{\end{ProofwCaption}}
\newenvironment{ProofwCaption}[1]
{\addvspace\theorempreskipamount \noindent{\it #1.}\rm}
{\qed \par \addvspace\theorempostskipamount}

\begin{document}
\title{Powers of Jacobi triple product, Cohen's numbers and the Ramanujan $\Delta$-function}
\author{Valery  Gritsenko
and  Haowu Wang\footnote{The study has been funded by the Russian Academic Excellence Project '5-100'. The second author was supported by Labex CEMPI, Lille.}}
\date{September 16, 2017}
\maketitle

\begin{abstract}
We show that the eighth power of the Jacobi triple product is a Jacobi--Eisenstein series of weight $4$ and index $4$ and we calculate its Fourier coefficients. As applications we obtain explicit formulas for the eighth  powers of theta-constants  of arbitrary order and the Fourier coefficients of the Ramanujan Delta-function
$\Delta(\tau)=\eta^{24}(\tau)$,
$\eta^{12}(\tau)$ and $\eta^{8}(\tau)$ 
in terms of Cohen's numbers $H(3,N)$ and $H(5,N)$.
We give new formulas for the number of representations of integers as sums of eight higher figurate numbers.
We also calculate the sixteenth and the twenty-fourth powers of the Jacobi theta-series using the basic Jacobi forms.

\end{abstract}

\section{Introduction}

In 1829, Jacobi found in \cite{J} explicit formulas for the number of 
representations by four, six and eight squares using  the theory of elliptic 
functions.
To do this he defined the theta-constants of order $2$ ($q=e^{2\pi i\tau}$)
$$
\theta_{00}(\tau)=\sum_{n=-\infty}^{+\infty}q^{\frac{n^2}{2} }, \quad\theta_{01}(\tau)= \sum_{n=-\infty}^{+\infty}(-1)^nq^{\frac{n^2}{2}},
$$
$$ 
\theta_{10}(\tau)=\sum_{n=-\infty}^{+\infty}q^{\frac{(2n+1)^2}{8}}
=2q^{\frac{1}{8}}\sum_{n\geq 0}q^{\frac{n(n+1)}{2}}.
$$
Then 
$$
\theta_{00}^k(2\tau)=1+\sum_{n=1}^{\infty}r_k(n)q^n, \quad \theta_{01}^k(2\tau)=1+\sum_{n=1}^{\infty}(-1)^nr_k(n)q^n
$$
and
$$
\theta_{10}^k(\tau)=2^kq^{\frac{k}{8}}\sum_{n=1}^{\infty}\delta_k(n)q^n
$$
where  $r_k(n)$ is the number of representations of $n$ as a sum of $k$ squares
and  $\delta_k(n)$ the number of representations of $n$ as a sum of $k$ triangular numbers, i.e. the number of non-negative $k$-tuples of integers 
$(x_1 ,\cdots, x_k )$ such that
$\frac{x_1(x_1+1)}{2}+\cdots +\frac{x_k(x_k+1)}{2}=n$.

Nowadays one knows that $\theta_{01}^8(2\tau)$ and $\theta_{10}^8(\tau)$ are modular forms of weight $4$ with respect to  the modular group $\Gamma_0(2)$. 
The graded ring of the $\Gamma_0(2)$-modular forms is polynomial. 
More exactly, it is generated by two Eisenstein series of weight $2$ and $4$
$$
M_*(\Gamma_0(2))=\bigoplus_{k\geq 0}M_{k}(\Gamma_0(2))=\CC[\varepsilon_2,\;\varepsilon_4]
$$
where $\varepsilon_2$ can be expressed using the quasi-modular Eisenstein series $E_2$
$$
\varepsilon_2(\tau)=2E_2(2\tau)-E_2(\tau)
=1+24\sum_{n\geq 1} \bigl(\sum_{d\vert n,\, d>0, \, {\rm odd}}d\,\bigr)\, q^n
$$
or as a special value of the Weierstrass $\wp$-function 
$\varepsilon_2(\tau)=-6\wp(\tau,\frac{1}2)$.

If 
$E_4(\tau)=1+240\sum_{n\geq 1}\sigma_3(n)q^n\in M_4(SL_2(\ZZ))$ then 
$M_4(\Gamma_0(2))$ is generated by $E_4(\tau)$ and  $E_4(2\tau)$. 
By comparing the first Fourier coefficients, we get 
$$ 
\theta_{01}^8(2\tau)=-\frac{1}{15}E_4+\frac{16}{15}E_4(2\tau), \quad 2^{-8}\theta_{10}^8(\tau)=\frac{1}{240}E_4-\frac{1}{240}E_4(2\tau).
$$
In particular we obtain the {\bf Jacobi formulas} (1829):
$$
r_8(n)=16\sum_{d\vert n, d>0}(-1)^{n+d}d^3, 
\qquad \delta_8(n)=\sum_{d\vert (n+1), \frac{n+1}{d}\, {\rm odd}} d^3.
$$

There are different generalisations of this classical subject. See, for example,
the paper \cite{Mi} of  S.~C.~Milne  who used analytical methods, and  
the paper \cite{ORW} of  K.~Ono, S.~Robins, P.T.~Wahl related to  the theory of  modular forms of half integral weights. 
Our results are based on the theory of Jacobi modular forms.

The three theta-constants written above are special values of the odd Jacobi theta-series $\vartheta$ of order $2$ (see \cite{Mu})
\begin{equation}\label{vartheta}
\vartheta(\tau,z)=
q^{\frac 1{8}}\zeta^{\frac1{2}}\sum_{n\in \ZZ}(-1)^{n}q^{\frac{n(n+1)}{2}}\zeta^{n}=
\sum_{n \in \ZZ}\  \bigl( \frac{-4}{n}\bigr)\, q^{\frac{n^2}{8}} \zeta^{\frac{n}{2}}
\end{equation}
where $\tau\in \HH=\{\tau\in\mathbb{C},\ {\rm Im}\,\tau>0\}$, $z\in \CC$, $q=\exp(2\pi i\, \tau)$,
$\zeta=\exp(2\pi i\, z)$ and $\bigl( \frac{-4}{n}\bigr)$ is the Kronecker symbol.
We note that
\begin{equation}\label{theta01} 
\vartheta(\tau,\frac{1}{2})=i\theta_{10},\quad 
\vartheta(\tau,\frac{\tau}{2})=-q^{-\frac{1}{8}} \theta_{01},\quad
\vartheta(\tau,\frac{\tau +1}{2})=iq^{-\frac{1}{8}}\theta_{00}.
\end{equation}
The Jacobi triple product formula
$$
\vartheta(\tau,z) = 
 -q^{\frac{1}{8}} \zeta^{-\frac{1}{2}} \prod_{n\geq 1} (1-q^{n-1}\zeta)(1-q^n \zeta^{-1})(1-q^n) 
$$
reflects the fact that $\vartheta(\tau,z)$ is the Kac--Weyl denominator function of
the simplest affine Lie algebra (see \cite{K}). This function was the basic Jacobi 
modular form in the classification of reflective Borcherds automorphic products 
and corresponding Lorentzian Kac--Moody algebras of Borcherds type in \cite{GN}.
Using the basic theta-series  $\vartheta(\tau,z)$ 
one can construct a lot of important examples of Jacobi modular forms in many variables 
and holomorphic Borcherds products
(see \cite{CG1}--\cite{CG2} and \cite{GPY}).

{\it In this paper we get the following generalisations of the Jacobi formulas
for $\theta_{00}^8$, $\theta_{10}^8$ $\theta_{01}^8$.} 
\smallskip

{\bf (1)} The eighth power of the Jacobi triple product
$$
\vartheta(\tau,z)^8=q\, \zeta^{-4}
\prod_{n\geq 1} (1-q^{n-1}\zeta)^8(1-q^n \zeta^{-1})^8(1-q^n)^8
$$
is a Jacobi--Eisenstein series. Its Fourier coefficients are expressed in terms 
of the Cohen numbers $H(3,N)$ (see Theorem \ref{Ftheta8}).
\smallskip

{\bf (2)} A similar result is true for the product $\wp(\tau,z)\vartheta(\tau,z)^8$
with the Weierstrass $\wp$-function (see Theorem \ref{Ftheta8}). Its Fourier coefficients are expressed in terms  of the Cohen numbers $H(5,N)$. 
\smallskip 

{\bf (3)} As applications we obtain a formula for the eighth power of theta-constants $\theta_{a,b}(\tau)$ of arbitrary order $N$ in terms of rational linear combinations of Cohen numbers multiplied by small powers of integers (\S 3.2).
This approach gives new formulas for the number of
representations of integers as sums of eight {\it higher figurate numbers}
with arbitrary parameter (see \S 3.3).
Moreover, we find new formulas for the Fourier coefficients 
$\tau(n)$ of the Ramanujan 
$\Delta$-function in terms of Cohen's numbers $H(3,N)$ and $H(5,N)$ and for 
the Fourier coefficients of $\eta^8(\tau)$, $\eta^{12}(\tau)$ 
(see \S 3.2, \S 3.4,  Lemma 2.1 and \S 4.1). 
\smallskip

{\bf (4)} We calculate explicitly cusp corrections of the powers
$\vartheta^{16}(\tau,z)$ and $\vartheta^{24}(\tau,z)$ and the products 
$\wp(\tau,z)^m\vartheta(\tau,z)^8$ for $m=2$, $3$ and $4$
(see \S 4.2 and \S 4.3). 
It gives new formulas for $\theta_{10}^{16}$ and  $\theta_{10}^{24}$.
\medskip

The subject of this work is closely related to the course of lectures 
``{\it Jacobi modular forms: 30 ans apr\`es}''
(see \cite{GJ}) given by  the first author in the Laboratory of Algebraic Geometry of  the Department of Pure Mathematics of HSE in Moscow in 2015--2016. 
In this course, an explicit formula for the eighth power of the Jacobi triple product was mentioned as a possible application.  
We express our  gratitude to all participants of the course, 
and especially to Dimitry Adler who was a research assistant of the course, and Guillaume Bioche who prepared a tex file of the lectures.

The authors are grateful to reviewers for many useful comments.

\section{Jacobi modular forms}\label{S-Jmf}

Consider a holomorphic function 
$\varphi:\mathbb{H}\times\mathbb{C}\rightarrow\mathbb{C}$.
Let  $k\in\mathbb{Z}$ be the weight and $ m\in\mathbb{N}$ the index.
We state the following two functional equations with 
$M=\left(\begin{smallmatrix} a & b \\ c & d \end{smallmatrix}\right)
\in SL_2(\mathbb{Z})$ and $\lambda,\mu\in\mathbb{Z}$:
\begin{align}
\varphi(\tau,z)=&
(c\tau+d)^{-k}e^{-2i\pi m\frac{cz^2}{c\tau+d}}
\varphi\left(\frac{a\tau+b}{c\tau+d},\frac{z}{c\tau+d}\right)
=(\varphi|_{k,m}[M])(\tau,z), 
\label{slash1}\\ 
\varphi(\tau,z)=&
e^{2i\pi m(\lambda^2\tau+2\lambda z)}\varphi(\tau,z+\lambda\tau+\mu)=
(\varphi|_{k,m}
\left([\begin{smallmatrix}\  \mu\\ -\lambda\end{smallmatrix};0]\right)
)(\tau,z) 
\label{slash2}
\end{align}
where we use the standard Jacobi slash operator ${}|_{k,m}$ (see \cite{EZ} or \cite{GN} where Jacobi forms of half-integral index $m$ were also considered)
and 
\begin{equation}\label{GammaJ}
[M]=\left(\begin{smallmatrix}
a&0&b&0\\ 
0&1&0&0\\
c&0&d&0\\
0&0&0&1\end{smallmatrix}\right)\in \Gamma^J, \quad
[\left(\begin{smallmatrix}\  \mu\\ -\lambda\end{smallmatrix}\right);0]
=\left(\begin{smallmatrix}
1&0&0&\ \mu\\ 
\lambda&1&\mu&\ 0\\
0&0&1&-\lambda
\\0&0&0&\ 1\end{smallmatrix}\right)\in \Gamma^J
\end{equation}
are $SL_2(\Bbb Z)$- and  unipotent generators of the Jacobi modular group $\Gamma^J$ which is a maximal parabolic subgroup of the integral symplectic group $Sp(2,\ZZ)$.
From the second equation one obtains immediately  that 
$\varphi(\tau,z+1)=\varphi(\tau,z)$. So a Jacobi form has a Fourier expansion: 
\[\varphi(\tau,z)=\sum_{n,\,l\in\mathbb{Z}}a(n,l)e^{2i\pi(n\tau+lz)}.\]

{\bf Definition} (see \cite{EZ}). {\it A holomorphic function $\varphi$
satisfying (\ref{slash1}) and (\ref{slash2}) is called a holomorphic (resp. cusp or   weak) Jacobi form of weight $k$ and index $m$ if $a(n,l)=0$ unless $4nm-l^2\geq0$ (resp. $4nm-l^2>0$ or $n\geq 0$)}.
\smallskip

We denote by $J_{k,m}$ (resp., $J_{k,m}^{\textrm{cusp}}$ or $J_{k,m}^{\rm w}$)
the finite-dimensional spaces of holomorphic Jacobi forms
(resp., Jacobi cusp forms or weak Jacobi forms).

\subsection{Jacobi theta functions of level two}

We recall the definitions of the Jacobi theta functions of level two
(see \cite[Chapter 1]{Mu}). We fix $q=e^{2\pi i \tau}$ and $\zeta=e^{2\pi i z}$, where $\tau \in \HH$ and $z\in \CC$. Then
\begin{align*}
&\vartheta_{00}(\tau,z)=
\sum_{n\in \ZZ}q^{\frac{n^2}{2}}\zeta^n,& 
&\vartheta_{01}(\tau,z)=\sum_{n\in \ZZ}(-1)^nq^{\frac{n^2}{2}}\zeta^n,\\
&\vartheta_{10}(\tau,z)=q^{\frac 1{8}}\zeta^{\frac 1{2}}
\sum_{n\in \ZZ}q^{\frac{n(n+1)}{2}}\zeta^{n},&  &
\vartheta_{11}(\tau,z)=iq^{\frac 1{8}}\zeta^{\frac 1{2}}
\sum_{n\in \ZZ}(-1)^nq^{\frac{n(n+1)}{2}}\zeta^{n}.
\end{align*}
We use another normalisation of the basic Jacobi theta-series 
$\vartheta(\tau,z) =-i\vartheta_{11}(\tau,z)$ (see (\ref{vartheta})). 
The last function is odd $\vartheta(\tau,-z)=-\vartheta(\tau,z)$ 
and $\Bbb Z\tau+\Bbb Z$ is  the set of its simple zeros.
We note that $\vartheta(\tau,z)$ is a Jacobi form of  weight $\frac 1{2}$ and index $\frac 1{2}$
with a multiplier system of order $8$ (see \cite{GN})
\begin{equation*}\label{thetaJ}
\vartheta(\tau,z)\in J_{\frac{1}{2},\frac{1}{2}}
(v_{\eta}^3 \times v_{H})
\end{equation*}
where $v_{\eta}$ is the multiplier system of the Dedekind $\eta$-function
and $v_H$ is a nontrivial binary character of the unipotent Heisenberg subgroup 
of the Jacobi modular group.
In particular $\vartheta^2(\tau,z)$ is a Jacobi form of weight $1$ and index $1$
with a character of order $4$ of the Jacobi modular group induced by 
$v_\eta^6\times 1_H$.

The Jacobi theta-series is a holomorphic Jacobi modular form of the minimal 
(singular) weight and the minimal positive index $1/2$.
It  plays an important role in the theory of Borcherds products
(see \cite{GN}, \cite{CG1},\cite{GPY}). 
For example, we have (see \cite{G1}) the following formulae 
for the generators of the graded ring of weak Jacobi modular forms 
$
J_{0,*}^{{\rm w},\ZZ}=\ZZ[\phi_{0,1}, \phi_{0,2}, \phi_{0,3}, \phi_{0,4}]
$
of weight $0$ and  integral indices  with {\it integral} Fourier coefficients
where 
\begin{align*}
\phi_{0,1}=& 4(\xi_{00}^2+\xi_{01}^2+\xi_{10}^2)=\zeta+10+\zeta^{-1}+q(\cdots) 
\in J^{w,\ZZ}_{0,1},\\
\phi_{0,2}=& 2(\xi_{00}^2\xi_{01}^2+\xi_{00}^2\xi_{10}^2+\xi_{10}^2\xi_{01}^2)
=\zeta+4+\zeta^{-1}+q(\cdots) \in J^{w,\ZZ}_{0,2},\\
\phi_{0,3}=& \left(\tfrac{\vartheta(\tau, 2z)}{\vartheta(\tau, z)}\right)^2
= \zeta+2+ \zeta^{-1}+q(\cdots) \in J^{w,\ZZ}_{0,3},\\
\phi_{0,4}=&\tfrac{\vartheta(\tau, 3z)}{\vartheta(\tau, z)}
 = \zeta+1+\zeta^{-1}+q(\cdots) \in J^{w,\ZZ}_{0,4}
\end{align*} 
and $\xi_{ab}(\tau,z)=\frac{\vartheta_{ab}(\tau,z)}{\theta_{ab}(\tau)}$
 (see (\ref{thetaab}) below). 
\smallskip

\noindent
The following remark is used throughout the article to demonstrate that various weak Jacobi forms are holomorphic Jacobi forms.

\noindent
{\bf Remark.} In this paper we often use the facts that
\begin{equation}\label{eta6phi01}
\eta^6\phi_{0,1},\quad \eta^3\phi_{0,2},\quad \eta^2\phi_{0,3}
\end{equation}
are holomorphic Jacobi forms (see \cite[Example 1.25]{GN}). 
The product  $\eta^2\phi_{0,4}$ is a Jacobi cusp form.

\subsection{Cohen's numbers}
The Henri Cohen numbers are generalisations of the Kronecker--Hurwitz
class number function $H(N)$ (see \cite{C} and \cite{Z}).

 Let $r$ and $N$ be non-negative integers and $r\geq 1$. 
For $N\geq 1$ we define
\begin{equation*}
h(r,N)= \begin{cases}
\tfrac{(-1)^{[r/2]}(r-1)!N^{r-\frac{1}{2}}}
{2^{r-1}\pi^{r}}\,
L(r,\chi_{(-1)^r N})  &\textit{if}\; (-1)^r N\equiv 0,1\mod 4,\\
0  &\textit{if}\; (-1)^r N\equiv 2,3\mod 4
\end{cases}
\end{equation*}
where $\chi_D=\left(\frac{D}{d} \right)$ is the Kronecker character. 
Then for non-negative $N\in \Bbb Q$
\begin{equation*}
H(r,N)= \begin{cases}
\sum_{d^2\vert N}h(r,N/d^2) &\textit{if}\; (-1)^r N\equiv 0,1\mod 4, \quad N>0,\\
\zeta(1-2r)  &\textit{if}\; N=0,\\
0 &\textit{otherwise}.
\end{cases}
\end{equation*}
The following properties are well known.

1) $H(r,N)$ is a rational number. For fixed $r$, the denominators of $H(r,N)$ are bounded. If $D=(-1)^r N$ is a discriminant of a quadratic field, then
$$
H(r,N)=L(1-r,\chi_D).
$$

2) If $(-1)^r N=Df^2$ with  discriminant $D$ of a quadratic field (we allow $D=1$ as a discriminant), we have 
$$
H(r,N)= L(1-r,\chi_D)\sum_{d\vert f}\mu(d)\chi_D(d) d^{r-1}\sigma_{2r-1}(f/d).
$$ 

3) For $r=1$, $H(1,N)=H(N)$ in the classical Kronecker--Hurwitz class number function.

\subsection{Jacobi forms and Jacobi Eisenstein series}
The Jacobi Eisenstein series (see \cite{EZ}) are defined by
\begin{equation*}
E_{k,m}(\tau,z)=\frac{1}{2}\sum_{\substack{(c,d)=1\\ c,d\in \ZZ}}\sum_{\lambda\in \ZZ}(c\tau+d)^{-k}e^m\left( \lambda^2\frac{a\tau+b}{c\tau+d}+2\lambda\frac{z}{c\tau+d} -\frac{cz^2}{c\tau+d} \right)
\end{equation*}
where $e^m(x)=e^{2\pi i mx}$. 
By \cite[Theorem 2.1]{EZ}, we know that  
the series $E_{k,m}$ ($k\geq 4$ even) converges and defines a non-zero element of $J_{k,m}$. The Fourier expansion is given by 
$$
E_{k,m}(\tau,z)=\sum_{\substack{n\in \NN,r\in\ZZ \\ 4nm\geq r^2}}e_{k,m}(n,r)q^n\zeta^r.
$$
For $4nm=r^2$, $e_{k,m}(n,r)=1$  if $r\equiv 0\mod 2m$ 
and $0$ otherwise. For $4nm>r^2$ we have 
\begin{equation*}\label{ek1}
e_{k,1}(n,r)=\frac{H(k-1,4n-r^2)}{\zeta(3-2k)}.
\end{equation*}
Moreover (see \cite[(13), page 48]{EZ}) 
\begin{equation}\label{Ekm}
E_{k,m}=m^{-k+1}\prod_{p\vert m}(1+p^{-k+1})^{-1}\sum_{d^2\vert m}\mu(d)E_{k,1}
\lvert_{k,m} U_d \circ V_{\frac{m}{d^2}}
\end{equation}
where $\left( \varphi\lvert_{k,m}U_d \right)(\tau,z)=\varphi(\tau,dz)$, and the operator 
$V_d$ maps $J_{k,m}$ to $J_{k,dm}$. If $\varphi\in J_{k,m}$, 
$\varphi=\sum a(n,r)q^n\zeta^r$, then 
\begin{equation}\label{Vl}
\varphi\lvert_{k,m}V_l=\sum_{n,r}\biggl(\sum_{d\vert (n,r,l)} d^{k-1}a\left(\frac{nl}{d^2},\frac{r}{d} \biggr) \right)q^n\zeta^r.
\end{equation}

It is easy to check that $E_{4,1}$, $E_{4,2}$, $E_{4,3}$, $E_{4,4}$, $E_{6,1}$, 
$E_{6,2}$, $E_{6,4}$ and $E_{8,1}$ have integral Fourier coefficients and
$E_{8,1}=E_4E_{4,1}$. In $\cite{EZ}$, there are tables for the first Fourier 
coefficients of $E_{4,1}$, $E_{6,1}$ and $E_{8,1}$.

Using GP PARI or considering the relations between Eisenstein series of Jacobi forms and Eisenstein series of Siegel modular forms of genus $2$
(see, for example, \cite{I} or the site \cite{S2}) 
one can calculate the first Fourier coefficients of $E_{10,1}$ and $E_{12,1}$. For $k=4,6,8,10$, we have 
$E_{k,m}(\tau,0)=E_k(\tau)$, where $E_k$ is the Eisenstein series of weight $k$ for $SL_2(\Bbb Z)$. $E_{12,1}(\tau,0)$ is related to the Ramanujan Delta function
$$
\Delta(\tau)=q\prod_{n\ge 1}(1-q^n)^{24}=\sum_{n\ge 1 }\tau(n)q^n.
$$
Using the fact that ${\rm dim}\, J_{10,1}={\rm dim}\, J_{12,1}=2$ we get

\begin{lemma}\label{E12} We have 
$$
E_{10,1}=1+q\left( \zeta^{\pm 2}-\tfrac{860776}{43867}\zeta^{\pm 1}
-\tfrac{9947070}{43867}   \right)+\cdots=\tfrac{23037}{43867}E_4E_{6,1}
+\tfrac{20830}{43867}E_6E_{4,1},
$$
$$
E_{12,1}=1+q\left( \zeta^{\pm 2}+\tfrac{339848}{77683}\zeta^{\pm 1}
+\tfrac{6971898}{77683} \right)+\cdots=\tfrac{27850}{77683}E_6E_{6,1}
+\tfrac{49833}{77683}E_4^2E_{4,1}.
$$
(In the sums above we denote $\zeta^{r}+\zeta^{-r}$ by $\zeta^{\pm r}$.)
In particular, 
$$
E_{12,1}(\tau,0)=E_{12}(\tau)+\tfrac{304819200}{53678953}\Delta(\tau),
$$
$$
\tau(n)=\tfrac{53678953}{304819200} 
\biggl( \sum_{4n\geq r^2,\, r\in \ZZ}\frac{H(11,4n-r^2)}{\zeta(-21)}-\frac{65520}{691}\sigma_{11}(n) \biggr).
$$
\end{lemma}

In \S 3.3 we give two formulas for $\tau(n)$ in the terms of  $H(3,N)$ and $H(5,N)$ using Jacobi--Eisenstein series of weight $4$ and $6$ together with new formulas for $\theta_{a,b}^8$ of theta-constants with rational characteristics.

\subsection{Pullbacks of Jacobi modular forms}

In this paper we consider the classical theta-constants $\theta_{a,b}(\tau)$
as specialisations of the basic Jacobi theta-series $\vartheta(\tau,z)$.
The next lemma from \cite{GJ} is a more explicit version of \cite[Theorem 1.3]{EZ}
(see also \cite[Lemma 2.1]{G1}).

\begin{proposition}\label{pullback}
Let $\varphi\in J_{k,m}$, $X=\binom{p}{-q} \in\mathbb{Q}^2$. 
We put  
$$
\Gamma_X=
\left\{M\in SL_2(\mathbb{Z})\,|\,MX\equiv X \mod \Bbb Z^2\right\}\supset \Gamma(N)
$$
where $N$ is the minimum positive integer such that $NX\in\Bbb Z^2$
and $\Gamma(N)$ is the principal congruence subgroup of level $N$.  
Then 
$$
\varphi_X(\tau)=e^{2i\pi m(q^2\tau+pq)}\varphi(\tau,q\tau+p)
\in M_k(\Gamma_X, \chi_X)
$$ 
is a modular form of weight $k$ with respect to $\Gamma_X$ with 
a character
$$
\chi_X(M)=e^{2i\pi m\textrm{det}(MX,X)}.
$$
\end{proposition}
\begin{proof} We use the $|_{k,m}$-action of $\Gamma^J$ defined in 
(\ref{slash1})--(\ref{slash2}). The two functional equations are equivalent
to the  equation $\varphi|_{k,m}g=\varphi$ for any $g\in \Gamma^J$.
Specialising (\ref{slash2}) for  $z=0$, we get 
$$
\varphi_X(\tau)=(\varphi|_{k,m}[X;0])|_{z=0}.
$$
For any $M\in SL_2(\mathbb{Z})$ we have 
$[X;0][M]=[M][M^{-1}X;0]$ and  
$$
\big(\varphi|_{k,m}[X;0]\big)|_{k,m}[M]=
\varphi|_{k,m}[M^{-1}X;0],\quad\forall\  M\in SL_2(\mathbb{Z}).
$$
We note that 
$
[X;\kappa_1][Y;\kappa_2]=[X+Y;\kappa_1+\kappa_2+\textrm{det}(X,Y)]
$.
Let $M\in\Gamma_X$, i.e. $MX-X\in\mathbb{Z}^2$. Then
$$
[M^{-1}X;0]=[M^{-1}X;0][X;0]^{-1}[X;0]=
[M^{-1}X-X;-\textrm{det}(M^{-1}X,X)][X;0].
$$
The Jacobi form $\varphi$ is $|_{k,m}$-invariant with respect to 
$[M^{-1}X-X;0]\in \Gamma^J$. For the central element of the Jacobi group
$[0;\kappa]$ with $\kappa\in \Bbb Q$ we have 
$\varphi|_{k,m}[0;\kappa]=e^{2i\pi m\kappa}\varphi$. 
Therefore we get
$$
\varphi_X|_kM
=e^{2i\pi m\,\textrm{det}(MX,X)}\varphi_X\quad\forall\, M\in\Gamma_X.
$$
One checks the holomorphicity (or cuspidality) at infinity using the conditions on the Fourier expansion of holomorphic (cusp) Jacobi forms.
\end{proof}

We get the following corollary from the proof of the last proposition.
\begin{corollary}\label{modular vector}
In Proposition \ref{pullback} we fix a finite system of representatives
$R_X=\{SL_2(\Bbb Z)X\}/\Bbb Z^2$. Then 
$\{\varphi_Y(\tau)\}_{Y\in R_X}$ is a vector valued modular form with respect to 
$SL_2(\Bbb Z)$. More exactly, for any $M\in SL_2(\Bbb Z)$ we have
$$
\varphi_S|_k\,M=e^{2i\pi m\, {\rm det}(MS_M, S)}\varphi_{S_M},\quad
$$
where $M^{-1}S \equiv S_M \mod \Bbb Z^2$ and $S_M\in R_X$.
\end{corollary}

\noindent
{\bf Remark.} If $X=\binom{1/2}{0}$ then $\Gamma_{X}=\Gamma_0(2)$. If the index 
$m$ is even then $\chi_{X}$ is trivial. If $m$ is odd then 
$\chi_{X}$ is a non-trivial binary character and $\phi_X(\tau)\in M_k(\Gamma_0(4))$.

If $X= {}^t(\tfrac{1}{2},-\tfrac{1}{2})$,  then
$\Gamma_{X}$ is the so-called theta-group
$$
\Gamma_\theta=
\biggl\{ M=\begin{pmatrix} a&b\\c&d\end{pmatrix} \in 
 SL_2(\ZZ)\,|\  M\equiv \begin{pmatrix} 1&0\\0&1\end{pmatrix} 
\ \text{ or }\ \begin{pmatrix} 0&1\\1&0\end{pmatrix}
\mod 2
\biggr\}
$$
which is a subgroup conjugate to $\Gamma_0(2)$.
\section{Explicit formulas without cusp terms} 

\subsection{Fourier expansions of $\vartheta^8(\tau,z)$}

\begin{theorem}\label{Ftheta8}
Let $f_4(n,r)$ be the Fourier coefficients of the eighth power of the Jacobi triple product
$$
\vartheta^8(\tau,z) = 
 q \zeta^{-4} \prod_{n\geq 1} (1-q^{n-1}\zeta)^8(1-q^n \zeta^{-1})^8(1-q^n)^8=
\sum_{n,r} f_4(n,r)q^n\zeta^r.
$$
Then, if $16n>r^2$, we have 
\begin{equation}\label{f4}
f_4(n,r)=-\frac{511}{2} H(3,\frac{16n-r^2}4)
+\frac{7}{2} \sum_{d \lvert (n,r,4)}d^3
H\left(3,\frac{16n-r^2}{d^2}\right)
\end{equation}
and 
$$
f_4(n,r)= \begin{cases}
1  &{\ if}\  16n=r^2\ {\ and\ } n\  {\ is\  odd},\\
0  &{\ if}\ 16n=r^2\  {\ and\ }  n\  {\ is\  even}.
\end{cases}
$$
(We note that $H(3,N)=0$ for any non integral $N$.)
The more compact form of the last identity is
\begin{equation}\label{theta8}
\vartheta^8(\tau,z)=E_{4,1}(\tau,2z)-E_{4,4}(\tau,z).
\end{equation}
Moreover,
\begin{equation}\label{wptheta8}
12\wp(\tau,z)\vartheta^8(\tau,z)=
\eta^6(\tau)\vartheta^6(\tau,z)\phi_{0,1}(\tau,z)
=E_{6,1}(\tau,2z)-E_{6,4}(\tau,z)
\end{equation}
where 
$$\wp(\tau,z)=-\frac{1}{4\pi^2z^2}-\frac{1}{4\pi^2}\sum_{\substack{\omega\in \ZZ\tau +\ZZ \\
\omega \neq 0}}\left((z+\omega )^{-2}-\omega^{-2}  \right)
\in J_{2,0}^{(mer)}$$
is the Weierstrass $\wp$-function which is  a meromorphic Jacobi form of weight $2$ 
and index $0$ with pole of order $2$ along $z\in  \ZZ\tau +\ZZ$.
\end{theorem}
\begin{proof} Since $\vartheta\in J_{\frac{1}{2},\frac{1}{2}}(v^3_\eta\times v_H)$ (see \S 2.1)
we have $\vartheta(\tau,z)^8\in J_{4,4}$.
We know that 
${\rm dim}\,J_{4,4}= {\rm dim}\, J_{6,4}=2$ and  
$J_{4,4}^{\rm cusp}=J_{6,4}^{\rm cusp}=\{0\}$.
Therefore $J_{4,4}$ and $J_{6,4}$ are generated by the Jacobi--Eisenstein seris
$E_{k,4}$ and $E_{k,1}(\tau, 2z)$ for $k=4$, $6$. 

By (\ref{ek1}), we have
\begin{equation*}
e_{4,1}(n,r) = \begin{cases}
-252H(3,4n-r^2)  &{\rm if}\; 4n>r^2,\\
\ \ 1  &{\rm if}\; 4n=r^2.
\end{cases}
\end{equation*}
(We note that $H(3,N)$ is negative.)
By (\ref{Ekm}), we get
\begin{equation*}
E_{4,4}(\tau,z)=\frac{1}{72}\bigl(E_{4,1}\lvert V_4-E_{4,1}(\tau,2z)\bigr).
\end{equation*}
By (\ref{Vl}), we obtain
\begin{equation*}
e_{4,4}(n,r) = \begin{cases}
\frac{7}{2}\bigl(H(3,\frac{16n-r^2}4)-\sum_{d \lvert (n,r,4)}
d^3H(3,\frac{16n-r^2}{d^2})\bigr) \ {\rm if}\; 16n>r^2,\\
\ 1\quad {\rm if}\; 16n=r^2, \; n \equiv 0\mod 2,\\
\ 0\quad  {\rm if }\;16n=r^2, \; n\equiv 1 \mod 2.
\end{cases}
\end{equation*}
We have 
\begin{align*}
E_{4,1}(\tau,2z)&=1+q(\zeta^{\pm 4}+56\zeta^{\pm 2}+126)+\dots,\\
E_{4,4}(\tau,z)&=1+q(8\zeta^{\pm 3}+28\zeta^{\pm 2}+56\zeta^{\pm 1}+56)+\dots.
\end{align*}
$\vartheta^8(\tau,z)$ has a vanishing constant Fourier coefficient $f_4(0,0)$ and its Fourier expansion starts with 
$q\zeta^{-4}(1-\zeta)^8$. Comparing non-trivial  Fourier coefficients
we get the formula (\ref{theta8}) which  implies the formula for the Fourier coefficients of $\vartheta^8(\tau,z)$.
\smallskip

We calculate the integral Fourier coefficients of $E_{6,4}$:
\begin{equation*}
e_{6,4}(n,r) = \begin{cases}
\frac{1}{8}\bigl(
H(5,\frac{16n-r^2}4)+\sum_{d \lvert (n,r,4)}-d^5H(5,\frac{16n-r^2}{d^2})\bigr)
\  {\rm if}\  16n>r^2,\\
1\ {\rm if}\  16n=r^2\ {\rm and}\ n\ {\rm is\  even},\\
0\ {\rm if}\  16n=r^2\ {\rm and}\ n\ {\rm is\  odd}.\\
\end{cases}
\end{equation*}
If $E_{6,1}(\tau,2z)-E_{6,4}(\tau,z)=\sum f_6(n,r)q^n\zeta^r$ then we have
\begin{equation*}
f_6(n,r) = \begin{cases}
1  &\textit{if}\; 16n=r^2 \ {\rm and}\ n\ {\rm is\  odd},\\
0  &\textit{if}\; 16n=r^2 \ {\rm and}\ n\ {\rm is\  even},
\end{cases}
\end{equation*}
\begin{equation}\label{f6}
f_6(n,r) =-\tfrac{1057}{8}H(5,\tfrac{16n-r^2}4)+\tfrac{1}{8} \sum_{d \lvert (n,r,4)}d^5 \cdot H\left(5,\tfrac{16n-r^2}{d^2}\right),\  16n>r^2.
\end{equation}
We note that $\phi_{0,1}=12\wp\vartheta^2/\eta^6$
and $\eta^6(\tau)\phi_{0,1}(\tau,z)$ is holomorphic at infinity
(see \cite[Example 1.25]{GN}). Therefore 
$\eta^6\phi_{0,1}\vartheta^6=12\wp\vartheta^8\in J_{6,4}$ and its constant Fourier coefficient vanishes. Comparing one non-zero Fourier coefficient we get the identity 
of the proposition. 
\end{proof}
\smallskip

\noindent
{\bf Remark.} In general, we can get more formulas by analysing higher indices and  higher weights. For instance, in the case of weight $4$, we have ${\rm dim} \, J_{4,m}=2$, for $m=8,9,12$. And we have that 
$\vartheta(\tau,3z)\vartheta^7(\tau,z) \in J_{4,8}$, 
$\vartheta^2(\tau,2z) \vartheta^6(\tau,z)\phi_{0,2} \in J_{4,9}$, 
$\vartheta^2(\tau,3z)\vartheta^6(\tau,z) \in J_{4,12}$  
and they do not have a $q^0$-term and as in the previous theorem, we obtain
\begin{align*}
\vartheta(\tau,3z)\vartheta^7(\tau,z)&=E_{4,8}(\tau,z)-E_{4,2}(\tau,2z),\\
\vartheta^2(\tau,2z) \vartheta^6(\tau,z) \phi_{0,2}&
=E_{4,1}(\tau,3z)-E_{4,9}(\tau,z),\\
\vartheta^2(\tau,3z)\vartheta^6(\tau,z)&=E_{4,3}(\tau,2z)-E_{4,12}(\tau,z).
\end{align*} 
We check that $\eta(\tau)^2\vartheta^2(\tau,z)\phi_{0,2}$ is a holomorphic Jacobi form. For that we need to calculate only its first $q$-term:
$q^\frac 13(r-2+r^{-1})(r+4+r^{-1})$.
The form $\vartheta^2(\tau,2z)\vartheta^6(\tau,z)\phi_{0,2}$
 is holomorphic as the product of the
last one  with the simplest  holomorphic theta-quark 
$\left(\frac{\vartheta(\tau,z)^2\vartheta(\tau,2z)}{\eta(\tau)}\right)^2$ (see, for example, \cite{CG2}).
\smallskip

\noindent
{\bf Jacobi theta-series and $\vartheta^8$.} We consider the even unimodular lattice $E_8$.
The Jacobi theta-series in one variable are examples of Jacobi modular forms
(see \cite{EZ}).  For any $u\in E_8$ one has
$$
\Theta_{E_8,u}(\tau,z)=\sum_{v\in E_8}\exp{\bigl(\pi i (v,v)\tau
+ 2 \pi i(v,u)z\bigr)}\in J_{4,\, \frac{(u,u)}2}.
$$
\begin{lemma}
We fix an element $u_2\in E_8$ with $(u_2,u_2)=2$ and a primitive vector 
$u_8\in E_8$ with $(u_8,u_8)=8$. Then 
$$
\vartheta^8(\tau,z)=\Theta_{E_8,u_2}(\tau,2z)-\Theta_{E_8,u_8}(\tau,z).
$$
\end{lemma}
\begin{proof} 
We note that $(u_2)^\perp_{E_8}=E_7$ and $(u_8)^\perp_{E_8}=A_7$.
The lattice $E_7$ (respectively, $A_7$) has $126$ (respectively, $56$) roots.
In the Fourier expansions above we have $e_{4,1}(1,0)=126$ and $e_{4,4}(1,0)=56$. Therefore
$$
E_{4,1}(\tau,z)=\Theta_{E_8,u_2}(\tau,z),
\quad
E_{4,4}(\tau,z)=\Theta_{E_8,u_8}(\tau,z).
$$
In particular, we see that the Fourier coefficients of Jacobi--Eisenstein series 
give the explicit formulas for the number of representations of natural integers
by quadratic forms of odd rank:
$$
e_{4,1}(n,0)=\#\{v\in E_7\,|\, (v,v)=2n\},\quad 
e_{4,4}(n,0)=\#\{v\in A_7\,|\, (v,v)=2n\}.
$$
We note that using the Zagier's $L$-function (see \cite{Z}) we found in 
\cite[Theorem 5.1]{GHS} wonderful formulas for the number of representations by any positive definite quadratic form
with one class in the genus (in particular, a new very short formula for the number of representations by five squares).
\end{proof}

\subsection{Application: $\theta_{a,b}^8(\tau)$}

The identity for $\vartheta^8(\tau,z)$ gives infinitely many explicit formulas for 
the eighth powers of the theta-constants $\theta_{a,b}^8$ with rational characteristics
(see below).

Let $(a,b)\in (\QQ^2\setminus \ZZ^2)\cup \{(0,0)\}$.
We recall the theta-series with characteristic $(a,b)$
\begin{equation}\label{thetaab}
\vartheta_{a,b}(\tau,z)
=\sum_{n\in \ZZ}e^{i\pi(n+a)^2\tau+2i\pi(n+a)(z+b)}.
\end{equation}
This is a holomorphic function on $\HH_1 \times\CC$. We put
$\theta_{a,b}(\tau)=\vartheta_{a,b}(\tau,0)$.

The classical Jacobi theta-series of order $2$ of Introduction and \S 2.1  have the following form in this  notation 
$\vartheta_{11}=\vartheta_{\frac 1{2}, \frac{1}2}$,
$\vartheta_{10}=\vartheta_{\frac 1{2}, 0}$, 
$\vartheta_{01}=\vartheta_{0, \frac{1}2}$ and 
$\theta_{10}=\theta_{\frac 1{2}, 0}$, 
$\theta_{01}=\theta_{0, \frac{1}2}$.
Among these series, there is a special one for $(a,b)=(0,0)$
$$
\vartheta_{00}(\tau,z)=\sum_{n\in \ZZ}e^{i\pi n^2\tau+2i\pi nz}
=\prod_{n\geqslant1}(1-q^n)(1+q^{\frac{2n-1}{2}}\zeta)(1+q^{\frac{2n-1}{2}}\zeta^{-1}).
$$
All the theta-series with characteristics can be expressed by means of 
$\vartheta_{00}$
$$
\vartheta_{a,b}(\tau,z)
=e^{2i\pi ab}q^{\frac{a^2}{2}}\zeta^{a}
\vartheta_{00}(\tau,z+{a}\tau+{b})=e^{i\pi ab}\vartheta_{00}|_{\frac{1}2,\frac{1}2}
[\left(\begin{smallmatrix}\ b\\ -a\end{smallmatrix}\right);0]
$$
(see \cite[(28)]{CG1} and the definition of the slash operator (\ref{slash2})). 
We will use the basic Jacobi form $\vartheta=-i\vartheta_{11}(\tau,z)$ 
since
$$
\vartheta|_{\frac{1}2,\frac{1}2}
[(\begin{smallmatrix}\ 1/2\\ -1/2\end{smallmatrix})]=
e^{\frac{3i\pi}{4}}\vartheta_{00}(\tau,z).
$$
Therefore we have
$$
\vartheta^8|_{4,4}
[\left(\begin{smallmatrix}\ b+1/2\\ -a-1/2\end{smallmatrix}\right);0]=
e^{4i\pi (a-b)}\vartheta^8|_{4,4}
[(\begin{smallmatrix}\ 1/2\\ -1/2\end{smallmatrix});0]
[\left(\begin{smallmatrix}\ b\\ -a\end{smallmatrix}\right);0]=
e^{4\pi i (-2ab+a-b)}\vartheta^8_{a,b}.
$$
In particular, using the notation of Proposition \ref{pullback}, we get
\begin{equation}\label{thetaabslash}
\vartheta^8_{a,b}(\tau,z)=e^{4\pi i(2ab-a+b)}\vartheta^8|_{X},
\qquad X=\left(\begin{smallmatrix}\ b+1/2\\ -a-1/2\end{smallmatrix}\right)
\end{equation}
or
\begin{equation*}
\theta^8_{a,b}(\tau)=e^{8\pi i(2ab+b)}q^{(1+2a)^2}
\vartheta^8(\tau,(a+\tfrac{1}{2})\tau+b+\tfrac{1}{2}).
\end{equation*}
Using the Fourier expansion of $\vartheta^8(\tau,z)$ from Theorem \ref{Ftheta8}
we obtain the Fourier expansion at infinity of
$\vartheta^8_{a,b}(\tau,z)$. We note that in this way we can get the Fourier 
expansions of $\theta^8_{a,b}(\tau)$ at {\bf all cusps} according to Corollary 
\ref{modular vector}.

\begin{corollary} 
We have the following formula for the eighth power of the Dedekind $\eta$-function
(see (\ref{f4}) of Theorem 3.1)
$$
\prod_{n=1}^{\infty}(1-q^n)^{8}=\sum_{n=0}^{\infty}\,
\biggl(\sum_{\substack{3m+2r+5=n \\ 16m\geq r^2}}f_4(m,r) \biggr)q^n.
$$
\end{corollary}
\begin{proof} By \cite[P.71]{Mu}, we have
$
\vartheta_{1/6,1/2}(3\tau)=e^{\frac{\pi i}{6}}\eta(\tau)
$.
Hence we proved the corollary. Moreover we obtain a Jacobi-Eisenstein type expression 
for this product
\begin{equation}\label{eta8}
q^{-\frac{1}{3}}\eta^8(\tau)=q^5\vartheta^8(3\tau,2\tau) =E_{4,1}(3\tau,\tau)-q^5E_{4,4}(3\tau,2\tau).
\end{equation}
\end{proof}

Next, we consider the special values of $\vartheta^8(\tau,z)$ in the  points $z$
of order two. 
For that we have to analyse the corresponding specialisations
of the Jacobi-Eisenstein series.
By Proposition \ref{pullback}, we have 
$E_{2k,2m}(\tau, \frac{1}2)\in M_{2k}(\Gamma_0(2))$.
For $2k=4$, $6$ this space contains only two Eisenstein series and no cusp forms.
The first coefficients show that  
\begin{align*}
E_{4,2}(\tau,\tfrac{1}2)&=E_{4,4}(\tau,\tfrac{1}2)=
\tfrac{16}{15}E_4(2\tau)-\tfrac{1}{15}E_4(\tau),\\
E_{6,2}(\tau,\tfrac{1}2)&=\tfrac{64}{63}E_6(2\tau)-\tfrac{1}{63}E_6(\tau),\\
E_{6,4}(\tau,\tfrac{1}2)&=\tfrac{127}{63}E_6(\tau)-\tfrac{64}{63}E_6(2\tau).
\end{align*}
By the previous formulas, we obtain the following identities for Cohen's numbers
for odd $n$:
\begin{align*}
&\sum_{\substack{8n>r^2\\ r \; even}}H(3,8n-r^2)=-4\sigma_3(n),&  &\sum_{\substack{8n>r^2\\ r \; odd}}H(3,8n-r^2)=-\frac{32}{7}\sigma_3(n),\\
&\sum_{\substack{8n>r^2\\ r \; even}}H(5,8n-r^2)=62\sigma_5(n),&  
&\sum_{\substack{8n>r^2\\ r \; odd}}H(5,8n-r^2)=64\sigma_5(n).
\end{align*}
In this way we obtain a new proof of  \cite[Corollary 5.4]{C} for any $n$.

For odd indices, we have  $E_{4,m}(\tau,\frac{1}2)\in M_4(\Gamma_0(4))$
(see the remark after Corollary 2.3).
One knows that $E_4(\tau)$, $E_4(2\tau)$, $E_4(4\tau)$ form a basis of 
$M_4(\Gamma_0(4))$. By the formulas for the Fourier expansions, we get
$$
E_{4,1}(\tau,\frac{1}2)=E_{4,3}(\tau,\frac{1}2)=\frac{1}{15}E_4(\tau)
-\frac{6}{5}E_4(2\tau)+\frac{32}{15}E_4(4\tau).
$$
Using
$$E_4(\tau)=E_{4,1}(\tau,0)=\sum_{n\geq 0}\left(-252\sum_{4n\geq r^2}
H(3,4n-r^2)\right)q^n,
$$
we get 
\begin{equation*}
\sum_{\substack{0<r<2\sqrt{n}\\ \textit{r is odd}}}H(3,4n-r^2)=-\frac{2}
{9}\sigma_3(n)-\frac{2}{7}\sigma_3(\frac{n}{2})+\frac{32}{63}\sigma_3(\frac{n}{4})
\end{equation*}
($\sigma_3(x)=0$ if $x$ is not an integer). The last formula was proved in \cite{C}
by another method.
\smallskip

For $z=\tfrac 12$  we get the Fourier expansion of $\theta_{10}^8(\tau)$
in terms of $H(3,N)$
\begin{align*}
\vartheta^8(\tau,\frac{1}2)=&\sum_{n\ge 1}\sum_{16n\ge r^2}
(-1)^rf_4(n,r)q^n=\theta_{10}^8(\tau)=2^8q\sum_{n> 0}\delta_8(n)q^n\\
=&E_4(\tau)-E_{4,4}(\tau,\frac{1}2)=
\frac{16}{15}(E_4(\tau)-E_4(2\tau)).
\end{align*}
The last expression gives  the Jacobi formula for $\theta_{10}^8(\tau)$ 
from the introduction. Moreover, when $n$ is odd and not a full square, 
we have
$$
2^8\delta_8(n-1)=\frac{7}{2}\sum_{16n> r^2}  (-1)^r H(3,16n-r^2)
-\frac{511}{2}\sum_{4n> r^2} H(3,4n-r^2).
$$
In the same way we get
\begin{align*}
\theta_{01}^8(2\tau)=&1+\sum_{n>0}(-1)^nr_8(n)q^n=q^2\vartheta(2\tau,\tau)^8=
E_4(2\tau)-q^2E_{4,4}(2\tau,\tau)\\
=&\sum_{n=0}^{\infty}\,
\biggl(\sum_{\substack{2m+r+2=n \\ 16m\geq r^2}}f_4(m,r) \biggr)q^n=-\frac{1}{15}E_4+\frac{16}{15}E_4(2\tau).
\end{align*}
Moreover, when $n$ is odd, we obtain a formula for the number of representations
by $8$ squares using $H(3,N)$ 
$$
r_8(n)=-\frac{7}{2}\sum_{\substack{2m+r+2=n \\ 16m> r^2}}H(3,16m-r^2).
$$

{\bf Remarks.} 
{\bf 1.} We note that for any non-zero rational characteristic 
$X={}^t(b,-a)$ the corresponding pullback $\wp(\tau, a\tau+b)$ is a modular form of weight $2$ with respect to the modular group $\Gamma_X$. Therefore the formula 
(\ref{wptheta8}) of Theorem 3.1 gives the  Fourier expansion
of $12\wp(\tau,a\tau+b)\theta_{a,b}(\tau)$.

{\bf 2.} For the points of order $2$ the values $\wp(\tau,a\tau+b)$ are the first 
generators
of weight $2$ of the corresponding graded rings of modular forms with respect to  
$\Gamma_0(2)$, $\Gamma^0(2)$ or $\Gamma_\theta$. For example, 
$
\varepsilon_2(\tau)=-6\wp(\tau, \frac{1}{2})
$
where 
$\varepsilon_2(\tau)$ is the Eisenstein series from the introduction
(see \cite{S1}).
Using the formula (\ref{wptheta8}) of Theorem 3.1 we get 
\begin{align*}
2\varepsilon_2(\tau)\vartheta^8(\tau,\frac{1}2)&
=\theta_{10}^8(\theta_{00}^4+\theta_{01}^4)=E_{6,4}(\tau,\frac{1}{2})-E_6(\tau)=\frac{64}{63}(E_6(\tau)-E_6(2\tau)),\\
\theta_{00}^8(\theta_{01}^4-\theta_{10}^4)&=E_6(\tau)-qE_{6,4}(\tau,\frac{\tau +1}{2}).
\end{align*}

{\bf 3.}
We note that the indentity (\ref{wptheta8}) of Theorem 3.1 yields a new representation
of the basic weak Jacobi form $2\phi_{0,1}$ which is equal to the elliptic genus of a K3 surface (see \cite{G1}). We hope that this representation might be useful in some moonshine constructions (see, \cite{E}).

\subsection{Application: representations of integers as sums of eight higher
figurate numbers}

The higher figurate numbers are defined by the functions ($a\geq 1$ and $n\in \ZZ$):
$$
f_a(n)=\frac{an^2+(a-2)n}{2}.
$$
Note that $a=1$ gives the triangular numbers and $a=2$ gives the squares
(see Introduction). It is a classical problem in number theory to define the number 
of representations of integral numbers by figurate numbers.
See \cite[\S 5]{ORW} where modular forms of half-integral weights are used.
 
Note that
\begin{align*}
&R_{a,m}(n)=\#\{(x_1,\ldots,x_m)\in \ZZ^m \,| \, n=f_a(x_1)+\ldots+f_a(x_m) \},\\
&R_{a,m}^{odd}(n)=\#\{(x_1,\ldots,x_m)\in \ZZ^m \,| \, n=f_a(x_1)+\ldots+f_a(x_m),
\ x_i \, {\rm is \  odd}\}.
\end{align*}
In this subsection, we give explicit formulas for $R_{a,8}$ and $R_{a,8}^{odd}$ for any $a$ by using Theorem 3.1. We have the following generating functions of the higher figurate numbers:
\begin{align*}
&\vartheta_{00}(a\tau,\frac{a-2}{2}\tau)=
\sum_{n\in\ZZ}e^{i\pi(n^2a\tau+n(a-2)\tau)}=
\sum_{n\geq 0}R_{a,1}(n)q^n,\\
&\vartheta_{10}(4a\tau,(a-2)\tau)=
\sum_{n\in\ZZ}e^{i\pi((n+\frac{1}2)^2 4a\tau+2(n+\frac{1}2)(a-2)\tau)}=
\sum_{n\geq 0}   R_{a,1}^{odd}(n)q^n
\end{align*}
where $\vartheta_{10}=\vartheta_{\frac{1}20}$ (see \S 3.2).
Therefore we have 
\begin{align*}
&\vartheta_{00}^m(a\tau,\frac{a-2}{2}\tau)=\sum_{n\geq 0}R_{a,m}(n)q^n,\\
&\vartheta_{10}^m(4a\tau,(a-2)\tau)=\sum_{n\geq 0}R_{a,m}^{odd}(n)q^n.
\end{align*}
We want to mention that our approach is different from \cite{ORW}. 
In \cite[Theorem 10]{ORW}, the generating function
$\vartheta_{00}(a\tau,\frac{a-2}{2}\tau)$ was represented as the quotient
$\frac{\eta(\tau)\eta_{a,1}(2\tau)}{\eta_{a,1}(\tau)}$
where $\eta_{a,1}(\tau)$ is a generalized Dedekind $\eta$-functions.

By (\ref{thetaab})--(\ref{thetaabslash}), we have
\begin{align*}
&\vartheta_{00}^8\left( a\tau,\frac{a-2}{2}\tau\right)=q^{3a-4}\vartheta^8
\left( a\tau,(a-1)\tau+\frac{1}{2} \right),\\
&\vartheta_{10}^8\left(4a\tau,(a-2)\tau\right)=\vartheta^8
\left(4a\tau,(a-2)\tau+\frac{1}{2} \right).
\end{align*}
Thus applying Theorem 3.1, we get the following result.
\begin{corollary}
For any positive integer $a$, we have
\begin{align*}
&R_{a,8}(n)=\sum_{\substack{am+r(a-1)+3a-4=n\\16m\geq r^2}}(-1)^rf_4(m,r),\\
&R_{a,8}^{odd}(n)=\sum_{\substack{4am+r(a-2)=n\\16m\geq r^2}}(-1)^rf_4(m,r).
\end{align*}
In particular, we have the following more familiar formulas.
\smallskip

\noindent
1) When $a$ is even and $n$ is odd
$$
R_{a,8}(n)=-\frac{7}{2}\sum_{\substack{am+r(a-1)+3a-4=n\\16m> r^2,\, r\  
{\rm is\, odd}}}H(3,16m-r^2).
$$
2) When $a$ is odd and $n$ is even
\begin{align*}
R_{a,8}(n)=&\frac{7}{2}\sum_{\substack{am+r(a-1)+3a-4=n\\16m> r^2,\  m\ 
{\rm is\ odd}}}(-1)^r  H(3,16m-r^2)\\
& -\frac{511}{2}\sum_{\substack{am+2s(a-1)+3a-4=n\\4m> s^2,\, m\ {\rm is\  odd}}}
H(3,4m-s^2)\\
&+\#\{(m,r)\in\ZZ^2\, : \, 16m=r^2,\, am+r(a-1)+3a-4=n\}.
\end{align*}
3) When $a$ is odd and $n$ is odd 
$$ 
R_{a,8}^{odd}(n)=-\frac{7}{2}\sum_{\substack{4am+r(a-2)=n\\16m> r^2,\,r\ 
{\rm  is\  odd}}}H(3,16m-r^2).
$$
\end{corollary}

\subsection{Application: $\tau(n)$}

Let us consider the automorphic correction of Jacobi forms 
(see \cite[Proposition 1.5]{G2}). 
Suppose that $\varphi \in J_{2k, m}$ has $a(0,0)=0$. We consider
$$  
 \operatorname{exp}(-8m\pi^2  G_2(\tau)z^2 )\varphi(\tau,z)= \sum_{n\geq 0}h_{n} (\tau ) z^{n}
$$
where $G_2(\tau)= -\frac{1}{24} + \sum_{n \geq 1} \sigma_1 (n) q^n $ is the quasi-modular Eisenstein series of weight $2$. Then we can check that the function 
$h_{n}(\tau)$ is  modular cusp form of weight $2k+n$ 
with respect to $SL_2(\ZZ)$. 
For $\varphi=\vartheta^8$ we have
\begin{gather*}
\exp(-32\pi^2  G_2(\tau)z^2 )\vartheta^8(\tau,z)=\\
\left( \frac{\partial^8 \vartheta^8}{\partial z^8}(\tau,0) \right)z^8 +\left( \frac{\partial^{10} \vartheta^8}{\partial z^{10}}-2880\pi^2G_2(\tau)\frac{\partial^8 \vartheta^8}{\partial z^8} \right)(\tau,0)z^{10} +\cdots .
\end{gather*}
Thus 
$$ \frac{\partial^8 \vartheta^8}{\partial z^8}(\tau,0)=a\Delta(\tau), \quad \frac{\partial^{10} \vartheta^8}{\partial z^{10}}(\tau,0)=2880\pi^2G_2(\tau)\frac{\partial^8 \vartheta^8}{\partial z^8}=b \Delta^{'}(\tau)$$ 
where $a,b$ are some constants. Hence we have
\begin{equation*}
\tau(n)= \tfrac{1}{8!}\sum_{r\in\ZZ, 16n\geq r^2}r^8f_4(n,r),\quad
n\tau(n)= \tfrac{3}{10!}\sum_{r\in\ZZ, 16n\geq r^2}r^{10}f_4(n,r).
\end{equation*}
In particular, when $n$ is odd and not an odd square, we have
\begin{equation}
\tau(n)=-\tfrac{73}{45}\sum_{r\in\ZZ, 4n> r^2}H(3,4n-r^2)r^8
+\tfrac{1}{11520}\sum_{r\in\ZZ, 16n> r^2}H(3,16n-r^2)r^8.
\end{equation}
In the same way, using (\ref{wptheta8}), we get a formula for $\tau(n)$ in terms of 
$H(5,N)$
\begin{equation*}
\tau(n)= \tfrac{1}{6!\cdot 12}\sum_{r\in\ZZ, 16n\geq r^2}r^6f_6(n,r),\quad
n\tau(n)= \tfrac{1}{8!\cdot 4}\sum_{r\in\ZZ, 16n\geq r^2}r^{8}f_6(n,r).
\end{equation*}
In particular, when $n$ is odd and is not an odd square, we have
\begin{equation*}
\tau(n)=-\tfrac{1057}{1080}\sum_{r\in\ZZ, 4n> r^2}H(5,4n-r^2)r^6+\tfrac{1}{69120}\sum_{r\in\ZZ, 16n> r^2}H(5,16n-r^2)r^6.
\end{equation*}

\section{Explicit formulas with cusp terms} 

\subsection{The Jacobi-Eisenstein series $E_{8,2}$, $E_{6,3}$ and $\eta^{12}(\tau)$}

The Jacobi--Eisenstein series  $E_{8,2}$ and $E_{6,3}$ are special because
${\rm dim}\,J_{8,2}^{\rm cusp}={\rm dim}\,J_{6,3}^{\rm cusp}=1$.
The corresponding Jacobi cusp forms are $\eta^{12}(\tau)\vartheta^4(\tau,z)$ and 
$\eta^6(\tau)\vartheta^6(\tau,z)$. 
Their values at $z=\frac{1}2$ are the first cusp forms 
$\eta^{12}\theta_{10}^4$ and $\eta^{6}\theta_{10}^6$ 
with respect to $\Gamma_0(2)$ and $\Gamma_0(4)$, respectively. 
We note that
\begin{equation*}
\eta^{12}(\tau)\theta_{10}^4(\tau)=2^4\eta^8(\tau)\eta^8(2\tau)
\quad{\rm and}\quad \eta^{6}(\tau)\theta_{10}^6(\tau)=2^6\eta^{12}(2\tau).
\end{equation*}

For any even index $E_{8,2m}(\tau,\frac{1}2)\in M_8(\Gamma_0(2))$ according to 
Proposition \ref{pullback} and we know that 
$(E_8(\tau), E_8(2\tau), \eta^{12}\theta_{10}^4)$ is a basis of 
$M_8(\Gamma_0(2))$.
Using  \S 2.3 we can find the Fourier expansions of $E_{8,2}$, $E_{8,4}$ and 
$E_{8,8}$. In particular,
$$ 
E_{8,2}(\tau,\frac{1}{2})=\sum_{n=0}^{\infty}\biggl(\sum_{r\in \ZZ,\,r^2\leq 8n}
\frac{(-1)^r}{129}\sum_{d\vert (n,r,2)} d^7 \frac{H(7,\frac{8n-r^2}{d^2} )}
{\zeta(-13)}\biggr)q^n. 
$$
By a direct calculation with Fourier expansions of Jacobi-Eisenstein series,
we get 
\begin{proposition}\label{E82}
The following formulas are valid
\begin{align*}
E_{8,2}(\tau, \tfrac{1}2)&=\tfrac{256}{255}E_8(2\tau)-\tfrac{1}{255}E_8(\tau)+
\tfrac{2160}{731}\,\eta^{12}\theta_{10}^4,\\
E_{8,4}(\tau, \tfrac{1}2)&=\tfrac{256}{255}E_8(2\tau)-\tfrac{1}{255}E_8(\tau)-
\tfrac{135}{731}\,\eta^{12}\theta_{10}^4,\\
E_{8,8}(\tau, \tfrac{1}2)&=\tfrac{256}{255}E_8(2\tau)-\tfrac{1}{255}E_8(\tau)+\tfrac{135}{731\cdot 16}\,\eta^{12}\theta_{10}^4.
\end{align*}
Let $\eta^{12}\theta_{10}^4=16q \prod_{n=1}^{\infty}(1-q^n)^8(1-q^{2n})^8
= \sum_{n=1}^{\infty}a(n)q^n$, then for any odd $n$ we have
$$ 
a(n)=\frac{86}{135}\sigma_7(n)+\frac{17}{6480}\sum_{r^2<8n}(-1)^r\frac{H(7,8n-r^2)}
{\zeta(-13)}.
$$
\end{proposition}

\smallskip

For any odd index $m$ we have $E_{6,m}(\tau,\frac{1}2)\in M_6(\Gamma_0(4))$ 
according to Proposition 2.2 and 
the remark after Corollary 2.3. One  knows that 
$E_6(\tau)$, $E_6(2\tau)$, $E_6(4\tau)$ and $\eta(2\tau)^{12}$ form a basis 
of $M_6(\Gamma_0(4))$. The direct calculation with the first Fourier coefficients 
yields
\begin{proposition}\label{E61}
The following formulas are valid
\begin{align*}
&E_{6,1}(\tau, \tfrac{1}2)=\tfrac{1}{63}E_6(\tau)-\tfrac{22}{21}E_6(2\tau)+\tfrac{128}{63}E_6(4\tau)-144\eta^{12}(2\tau),\\
&E_{6,3}(\tau, \tfrac{1}2)-E_{6,1}(\tau, \tfrac{1}2)=\tfrac{9216}{61}\eta^{12}(2\tau).
\end{align*}
In particular, let $\eta^{12}(2\tau)=q\prod_{n=1}^{\infty}(1-q^{2n})^{12}= \sum_{n=0}^{\infty}b(n)q^{n}$, then for any $n\in \NN$ we have
\begin{align*}
&b(2n+1)=\frac{11}{12}\sum_{8n+4\geq r^2}(-1)^rH(5,8n+4-r^2)-\frac{1}{18}\sigma_5(2n+1),\\
&b(2n)=0: \quad \sum_{\substack{0<r<2\sqrt{2n}\\ \textit{r is odd}}}H(5,8n-r^2)=\frac{31}{33}\sigma_5(2n)+\sigma_5(n)-\frac{64}{33}\sigma_5(\frac{n}{2}).
\end{align*}
(We have $\sigma_5(\frac{n}{2})=0$ for any odd $n$.)
\end{proposition}

In the same way we can analyse $E_{6,3}$. 
We have 
$$
E_{6,3}(\tau,\frac{1}3)\in M_{6}(\Gamma_1(3))=
\latt{E_6(\tau),\  E_6(3\tau),\  (\eta(\tau)\eta(3\tau))^6}_\CC
$$  
where  $(\eta(\tau)\eta(3\tau))^6$ is the first cusp form with respect to $\Gamma_1(3)$ and $\Gamma_0(3)$. 
Comparing the Fourier coefficients we get

\begin{proposition}\label{E66}
$$E_{6,3}(\tau,\tfrac{1}3)=-\tfrac{1}{728}E_6(\tau)+\tfrac{729}{728}E_6(3\tau)-\tfrac{28512}{793}(\eta(\tau)\eta(3\tau))^6.
$$
We put  $(\eta(\tau)\eta(3\tau))^6=\sum_{n\geq 1}c(n)q^n$.  Then
\begin{multline*}
c(n)=\tfrac{61}{3168}\sigma_5(n)-\tfrac{4941}{352}\sigma_5(\frac{n}{3})\\
+\tfrac{13}{864}\sum_{r^2\leq 12n}\cos(\frac{2\pi  r}{3}) \sum_{d\vert (n,r,3)}d^5 H(5,\frac{12n-r^2}{d^2}).
\end{multline*}
\end{proposition}

\subsection{An explicit formula for $\vartheta^{16}(\tau,z)$}

In \S 2.1 we give the formulas for the generators of the graded ring
$J_{0,*}^{{\rm w},\ZZ}$. We can use these generators for the construction of cusp corrections 
in some explicit formulas similar to the formulas of Theorem \ref{Ftheta8}.

\begin{theorem}The following identities are true
\begin{align*}
(12\wp)^2\vartheta^8=&\eta^{12}\vartheta^4\phi_{0,1}^2=E_{8,1}(2z)-E_{8,4}
+\tfrac{1449}{86}\,\eta^{12}\vartheta^4\phi_{0,2},\\
(12\wp)^3\vartheta^8=&\eta^{18}\vartheta^2\phi_{0,1}^3=E_4E_{6,1}(2z)-E_6E_{4,4}+36\Delta(\phi_{0,1}\phi_{0,2} -2\phi_{0,3}),\\
(12\wp)^4\vartheta^8=&\Delta\phi_{0,1}^4=E_6E_{6,1}(2z) - E_8E_{4,4}+
48\Delta(\phi_{0,1}^2\phi_{0,2}-9\phi_{0,1}\phi_{0,3}+12\phi_{0,4}),
\end{align*}
\begin{multline}\label{theta16}
\vartheta^{16}=E_{8,2}(2z)-E_{8,8}+\\
\eta^{12}\vartheta^4\left(\tfrac{73}{11008}\phi_{0,1}\phi_{0,2}\phi_{0,3}-
\tfrac{45549}{2752}\phi_{0,3}^2+\tfrac{20713}{1376}\phi_{0,2}\phi_{0,4} \right).
\end{multline}
\end{theorem}
\begin{proof}
We know (see \cite{EZ}) that 
${\rm dim}\,J_{k,4}^{\rm Eis}=2$ for any even $k\ge 4$ and ${\rm dim}\,J_{8,4}^{\rm cusp}=1$, 
${\rm dim}\,J_{10,4}^{\rm cusp}=2$, 
and ${\rm dim}\,J_{12,4}^{\rm cusp}={\rm dim}\,J_{8,8}^{\rm cusp}=3$. 
Using the conditions of holomorphicity  (\ref{eta6phi01}) 
for the four generators 
$\phi_{0,1}$, ..., $\phi_{0,4}$ of weight $0$, we can find a  basis for any  of these spaces of cusp forms:
\begin{align*}
&J_{8,4}^{\rm cusp}= \latt{ \eta^{12}\vartheta^4\phi_{0,2}}_{\CC},\\
&J_{10,4}^{\rm cusp}= \latt{ \eta^{18}\vartheta^2\phi_{0,1}\phi_{0,2},\; \eta^{18}\vartheta^2\phi_{0,3}}_{\CC},\\
&J_{12,4}^{\rm cusp}= \latt{ \Delta\phi_{0,1}^2\phi_{0,2},\; \Delta\phi_{0,1}\phi_{0,3},\;\Delta\phi_{0,4}}_{\CC},\\
&J_{8,8}^{\rm cusp}= \latt{\eta^{12}\vartheta^4\phi_{0,1}\phi_{0,2}\phi_{0,3},\;\eta^{12}\vartheta^4\phi_{0,3}^2,\;\eta^{12}\vartheta^4\phi_{0,2}\phi_{0,4}}_{\CC}.
\end{align*}
Firstly, the above functions are Jacobi cusp forms by (\ref{eta6phi01}).
We know that the functions $\phi_{0,1}, \phi_{0,2}, \phi_{0,3}$ are algebraically independent over $\CC$ and $4\phi_{0,4}=\phi_{0,1}\phi_{0,3}-\phi_{0,2}^2$. 
Thus we get a basis in all cases above.

We obtain  the  identities of the theorem by comparing the first Fourier coefficients
with GP PARI.
\end{proof}
\smallskip

{\bf Applications to $\theta_{10}^{16}$ and $\theta_{01}^{16}$.}
The formula (\ref{theta16}) for $\vartheta^{16}\in J_{8,8}$ looks complicated but it 
immediately  gives explicit formulas for the $16$-powers of the theta-constants of order $2$.
According to \cite{G1} the special values  of three generators
of weight $0$ at $z=\frac 12$ are constants and $\phi_{0,1}(\tau,\frac{1}{2})$ is a modular function of weight $0$:
$$
\phi_{0,2}(\tau,\frac{1}{2})=2,\quad
\phi_{0,3}(\tau,\frac{1}{2})=0,\quad
\phi_{0,4}(\tau,\frac{1}{2})=-1,
$$
$$
\phi_{0,2}(\tau,\frac{\tau +1}{2})=-2q^{-\frac{1}{2}},\ 
\phi_{0,3}(\tau,\frac{\tau +1}{2})=0,\ 
\phi_{0,4}(\tau,\frac{\tau +1}{2})=-q^{-1},
$$
$$
\phi_{0,1}(\tau,\frac{1}{2})=4\left(\frac{\theta_{01}^2}{\theta_{00}^2} +
\frac{\theta_{00}^2}{\theta_{01}^2}  \right),\quad
\phi_{0,1}(\tau,\frac{\tau +1}{2})=4q^{-\frac{1}{4}} \left(\frac{\theta_{10}^2}
{\theta_{01}^2} -\frac{\theta_{01}^2}{\theta_{10}^2}  \right).
$$
By (\ref{theta16}) and Proposition \ref{E82}, we obtain three different formulas
for $\theta^{16}_{10}$:
\begin{align*}
\theta^{16}_{10}(\tau)&=E_{8}(\tau)-E_{8,8}(\tau, \tfrac{1}2)-
\tfrac{20713}{688}\,\eta^{12}(\tau)\theta_{10}^4(\tau),\\
\theta^{16}_{10}(\tau)&=\tfrac{256}{255}\left( E_{8}(\tau)-E_{8}(2\tau) \right)-
\tfrac{512}{17}\eta^{12}(\tau)\theta_{10}^4(\tau),\\
\theta^{16}_{10}(\tau)&=\tfrac{1952}{2025}E_8(\tau)+\tfrac{18688}{2025}E_8(2\tau)-\tfrac{1376}{135}E_{8,2}(\tau, \frac{1}2).
\end{align*}
Therefore we have 
$$
\delta_{16}(n)=\tfrac{61}{8640}\sigma_{7}(n+2)-\tfrac{1}{829440}
\sum_{8(n+2)>r^2}(-1)^r\frac{H(7,8(n+2)-r^2)}{\zeta(-13)}
\quad (n\ {\rm is\   odd}).
$$
In the same way we get 
$$
\theta^{16}_{01}(2\tau)=-\tfrac{13}{2025}E_8(\tau)+\tfrac{3328}{2025}E_8(2\tau)-\tfrac{86}{135}E_{8,2}(\tau, \frac{1}2)
$$
and
$$
r_{16}(n)=\tfrac{416}{135}\sigma_{7}(n)+
\tfrac{2}{405}\sum_{8n>r^2}(-1)^r\frac{H(7,8n-r^2)}{\zeta(-13)}
\quad (n\ {\rm is\  odd}).
$$

\subsection{An explicit formula for $\vartheta^{24}(\tau,z)$}
We use the same method as above.
We know (see \cite{EZ}) that 
$$ 
{\rm dim}\, J_{12,12}^{\rm Eis}=2,\qquad 
{\rm dim}\,J_{12,12}^{\rm cusp}=9.
$$
A system of generators  of $J_{12,12}^{\rm cusp}$ 
can be constructed by basic weak  Jacobi forms 
of weight zero:
\begin{align*}
&\Delta\phi_{0,1}\phi_{0,2}\phi_{0,3}^3,& &\Delta\phi_{0,1}\phi_{0,2}^4\phi_{0,3},& &\Delta\phi_{0,2}^3\phi_{0,3}^2,\\
&\Delta\phi_{0,2}^6,& &\Delta\phi_{0,3}^4,& &\Delta\phi_{0,1}^2\phi_{0,3}^2\phi_{0,4},\\
&\Delta\phi_{0,1}\phi_{0,2}^2\phi_{0,3}\phi_{0,4},& &\Delta\phi_{0,1}^2\phi_{0,2}\phi_{0,4}^2,& &\eta^{12}\vartheta^{12}\phi_{0,1}\phi_{0,2}\phi_{0,3}.
\end{align*}
First we can check that the above functions are Jacobi cusp forms
(see (\ref{eta6phi01})). 
We can analyse their first Fourier coefficients
\begin{align*}
&\eta^{12}\vartheta^{12}\phi_{0,1}\phi_{0,2}\phi_{0,3}=q^2(\zeta^{\pm 9}+\cdots)+q^3(-\zeta^{\pm 11}+\cdots)+\cdots,\\
&\Delta\phi_{0,1}^2\phi_{0,2}\phi_{0,4}^2=q(*\zeta^{\pm 6}+\cdots)+q^2(*\zeta^{\pm 8}+\cdots)+q^3(\zeta^{\pm 11}+\cdots)+\cdots.
\end{align*}
The rest of the functions have a Fourier expansion of the type
$$
q(*\zeta^{\pm 6}+\cdots)+q^2(*\zeta^{\pm 8}+\cdots)+q^3(*\zeta^{\pm 10}+\cdots)+\cdots.
$$
Their Fourier coefficients at $q\zeta^6$, $q\zeta^5$, $q\zeta^4$, 
$q\zeta^3$, $q\zeta^2$, $q\zeta^1$ and  $q^2\zeta^8$ form a matrix $M$
with det\,$M\neq 0$.

The Jacobi--Eisenstein series of weight $12$ and index $12$ has rational Fourier
coefficients with rather large denominators (see Lemma \ref{E12}).
Thus we use the Eisenstein series with integral coefficients.
By GP PARI, we obtain
\begin{multline*}
\vartheta^{24}(\tau,z)=E_4^2E_{4,3}(\tau,2z)-E_{4,4}^3
-24\Delta\phi_{0,1}^2\phi_{0,2}\phi_{0,4}^2+36\Delta\phi_{0,2}^6\\
+\Delta\phi_{0,3}\bigl(-38\phi_{0,1}\phi_{0,2}^4
-477\phi_{0,1}\phi_{0,2}\phi_{0,3}^2
+486\phi_{0,2}^3\phi_{0,3}
+702\phi_{0,3}^3\\
+55\phi_{0,1}^2\phi_{0,3}\phi_{0,4}
+160\phi_{0,1}\phi_{0,2}^2\phi_{0,4}\bigr).
\end{multline*}
We note that in the right hand side of the last formula only two cusp corrections do
not vanish for $z=\frac{1}2$! Thus  we have
\begin{equation*}
\theta_{10}^{24}(\tau)=E_4^3(\tau)-E_{4,4}^3(\tau,\frac{1}{2})
-48\Delta(\tau)\phi_{0,1}^2(\tau,\frac{1}{2})+2^8\cdot 3^2\Delta(\tau).
\end{equation*}

By using formula for $\phi_{0,1}(\tau,\frac{1}{2})$ and the classical relation $2\eta^3=\theta_{00}\theta_{01}\theta_{10}$, we can simplify the above formula
\begin{equation}
\theta_{10}^{24}=E_4^3-E_{4,4}^3(\tau,\frac{1}{2})
-3\cdot 2^4\,\eta^{12} \theta_{10}^4E_4.
\end{equation}
Hence we get an expression of $\theta_{10}^{24}$ in terms of $E_{8,m}(\tau,\frac 1{2})$:
\begin{gather}
\theta_{10}^{24}=E_4^3-E_{4,4}^3(\tau,\tfrac{1}{2})
-\tfrac{688}{45}E_4 \left(E_{8,2}(\tau,\tfrac{1}{2})
-E_{8,4}(\tau,\tfrac{1}{2})   \right),
\end{gather}
since we have $\eta^{12} \theta_{10}^4=\tfrac{43}{135}\left(E_{8,2}(\tau,\tfrac{1}{2})-E_{8,4}(\tau,\tfrac{1}{2})   \right)$ by Proposition 4.1.

Valery Gritsenko \par
\smallskip

Laboratoire Paul Painlev\'e\par
Universit\'e de Lille 1 and IUF, France \par
\smallskip

\par
National Research University Higher School of Economics
\par
Russian Federation
\par
\smallskip
Valery.Gritsenko@math.univ-lille1.fr
\vskip0.5cm

\vskip5pt

Haowu Wang \par
\smallskip

Le Laboratoire d'Excellence Centre Europ\'een pour les Math\'ematiques,
\par
la Physique et leurs interactions (CEMPI)
\par
Universit\'e de Lille 1, France
\par
\smallskip
haowu.wang@math.univ-lille1.fr
\end{document}